\documentclass[12pt,a4paper]{amsart}

\usepackage{latexsym, amssymb,amsmath, amsthm,amsfonts, mathrsfs}
\usepackage[arrow, matrix, curve]{xy}
\usepackage{hyperref}
 \usepackage{verbatim}

\usepackage{nicefrac}

\newcommand{\CC}{\mathbb{C}}

\newcommand{\QQ}{\mathbb{Q}}

\newcommand{\NN}{\mathbb{N}}
\newcommand{\kk}{\Bbbk}
\newcommand{\graph}{\Gamma_{G}}

\newcommand{\graphD}{\Gamma_{D}}
\newcommand{\pl}{\operatorname{pl}}
\newcommand{\SG}{\mathcal{S}_{G}}
\newcommand{\SGa}{\mathcal{S}_{\Ga}}
\newcommand{\SD}{\mathcal{S}_{D}}

\newcommand{\kvn}{\kk[V_n]}
\newcommand{\kvndn}{\kk[V_{n}]^{D_{n}}}
\newcommand{\graphDn}{\Gamma_{D_n}}
\newcommand{\SDn}{\mathcal{S}_{D_n}}

\newcommand{\kxd}{{\kk[X]^D}}
\newcommand{\kx}{{\kk[X]}}
\newcommand{\kxg}{{\kk[X]^G}}

\newcommand{\done}{\hfill $\triangleleft$}

\newcommand{\Ga}{{{\GG_a}}}

\newcommand{\V}{{\mathcal{V}}}

\newcommand{\s}{\mathcal{S}}

\def\SL2{\operatorname{SL}_{2}(K)}

\def\GL2{\operatorname{GL}_{2}(K)}
\def\Ga{{\mathbb G}_{a}}

\def\INVSL2{$K[V]^{operatorname{SL}_{2}(K)}$}
\def\INVSO2{$K[V]^{operatorname{SO}_{2}(K)}$}
\def\INVGL2{$K[V]^{operatorname{GL}_{2}(K)}$}

\def\GL{\operatorname{GL}}
\def\diag{\operatorname{diag}}
\def\SL{\operatorname{SL}}

\def\diag{\operatorname{diag}}

\newtheorem{lem}{Lemma}[section]
\newtheorem{thm}[lem]{Theorem}

\newtheorem{cor}[lem]{Corollary}

\newtheorem{prop}[lem]{Proposition}

\theoremstyle{definition}

\theoremstyle{remark}
\newtheorem{rmk}[lem]{Remark}
\newtheorem{eg}[lem]{Example}

\newtheoremstyle{Acknowledgements}
  {}
    {}
     {}
     {}
    {\bfseries}
    {}
     {.5em}
     {\thmname{#1}\thmnumber{ }\thmnote{ (#3)}}
\theoremstyle{Acknowledgements}
\newtheorem{ack}{Acknowledgements.}

\title[The separating variety for the basic $\Ga$-actions]{The separating variety for \\the  basic representations of\\ the  additive group}

\author{ Emilie Dufresne}
\address{Mathematisches Institut\\
Universit\"at Basel\\
Rheinsprung 21\\
4051 Basel, Switzerland}
\email{emilie.dufresne@unibas.ch}

\author{ Martin Kohls}
\address{Technische Universit\"at M\"unchen \\
 Zentrum Mathematik-M11\\
Boltzmannstrasse 3\\
 85748 Garching, Germany}
\email{kohls@ma.tum.de}

\date{\today}



\begin{document}
\maketitle


\begin{abstract}
For a group $G$ acting on an affine variety $X$,  the
separating variety  is the closed subvariety of $X\times X$  encoding which
points of $X$ are separated by invariants.  We concentrate on the
indecomposable rational linear representations $V_n$ of dimension $n+1$ of the additive group  of a field of characteristic zero,
and decompose the separating variety into the union of irreducible components. We show that if $n$ is odd, divisible by four, or equal to two,
the closure of the graph of the action, which has dimension $n+2$, is the only component of the
separating variety. In the remaining cases, there is a second irreducible
component of dimension $n+1$. We conclude that in these cases, there are no
polynomial separating algebras.
 \end{abstract}



\section{Introduction}

Let $\kk$ be an algebraically closed field, and let $G$ be an algebraic  group acting rationally on an irreducible affine variety $X$. This action induces an action on $\kx$, the ring of regular functions on $X$, via $(\sigma * f) (u)=f(\sigma^{-1}* u)$. The \emph{ring of invariants} is the subalgebra $\kxg\subseteq\kx$ formed by the elements fixed by $G$, or equivalently, the subalgebra formed by the elements which are constant on the orbits. Thus, for $x,y\in X$ and $f\in\kxg$, having $f(x)\neq f(y)$ implies that $x$ and $y$ belong to distinct orbits. In this situation, we say that the invariant $f$ \emph{separates} $x$ and $y$. A \emph{separating set} is a set of invariants which separate any two points which are separated by some invariant (see \cite[Definition 2.3.8]{hd-gk:cit}).

The \emph{separating variety}
\[
\SG:=\{(x,y)\in X\times X \mid f(x)=f(y) \text{ for all } f\in \kxg\}
\]
provides an alternative characterization of separating sets. Namely, if $\delta:\,\,\kx\rightarrow \kx\times \kx$ is the map defined by
$\delta(f):=f\otimes 1-1\otimes f$, then $E\subseteq \kxg$ is a separating set if and only if $\V_{X\times
  X}(\delta(E))=\SG=\V_{X\times  X}(\delta(\kxg))$, where $\V$ denotes the
common zero set of a set of polynomials. The separating variety
encodes which points can be separated using invariants.  In the case
of finite groups, the invariants separate the orbits, and so the
separating variety is in fact equal to the  \emph{graph} of the
$G$-action:
\[
\graph:=\{(x,\sigma\cdot x))\in X\times X\mid x\in X,~ \sigma\in G\}.
\]
This fact played a central role in the proof that, when $X$
is a representation of a finite group $G$, if there exists a polynomial separating algebra, then the action of $G$ on $X$ is generated by reflections (see \cite[Theorem 1.1]{ed:sifrg}).

The graph consists of those pairs of points which belong to the same
orbit, while the separating variety consists of the pairs of points
which can not be separated by invariants. Thus, we always have
$\graph\subseteq\SG$. Moreover, as $\SG$ is Zariski-closed, we also
have $\overline{\graph}\subseteq \SG$. Even for reductive groups, this inclusion can be strict  (see \cite[Example
2.1]{gk:cirgpc}). The invariants
may not always separate orbits  (as for the natural action of the
multiplicative group on a vector space), but in
the case of reductive groups, they do separate disjoint orbit
closures (see \cite[Corollary 3.5.2]{Newstead}). Exploiting this,
Kemper gives an algorithm to compute the separating variety and then
a separating set (see \cite[Algorithm 2.9]{gk:cirgpc}), which is  the
first step in his algorithm to compute the invariants of reductive
groups in arbitrary characteristic (see \cite[Algorithm
1.9]{gk:cirgpc}).

The motivation for this paper is to better understand the separating variety
in the case of non-reductive groups. We concentrate on what is perhaps the
simplest situation:  algebraic actions of the additive group $\Ga=(\kk,+)$ on
an irreducible affine variety $X$, where $\kk$ is a field of characteristic zero.

Actions of the additive group on $X$ are in one to one correspondence with
locally nilpotent derivations (abbreviated LND) on $\kx$. Recall that a \emph{locally nilpotent derivation} $D$ is a $\kk$-linear map $\kx\rightarrow\kx$ such that $D(ab)=aD(b)+bD(a)$ for all $a,b\in\kx$ and, for all $a\in\kx$, there exists an $m\geq 1$ such that $D^m(a)=0$. A locally nilpotent derivation $D$ on $\kx$ induces an action $*:
\Ga\times \kx\rightarrow \kx$ via
\[
(-t)*f:=\exp(tD)f=\sum_{k=0}^{\infty}\frac{t^{k}}{k!}D^{k}(f)\text{  for }t\in \Ga , ~ f\in \kx.
\]
The invariant ring $\kx^{\Ga}$ coincides with the kernel of $D$ and is denoted by $\kxd$. We write $\SD=\SGa$ to denote the separating variety
corresponding to the action induced by the locally nilpotent derivation $D$, and $\Gamma_D$ to denote the graph of the corresponding $\Ga$-action.

An important contribution of the LND approach is van den Essen's algorithm
to compute the kernel of a LND, and thus the invariants of a $\Ga$-action
(see \cite{ae:aacigaav}). An element $s\in\kx$ such that $Ds\ne 0$ and
$D^{2}s=0$ is a  \emph{local slice}. By the Slice Theorem (which is in fact
the first step of the algorithm, see \cite[Section 3]{ae:aacigaav}), for a
local slice $s$ and any $f\in \kx$, the element
$\pi(f):=\exp(tD)f|_{t:=-\nicefrac{s}{Ds}}$ is in $\kx^{D}_{Ds}$, and the
algebra homomorphism $\pi$ maps $\kx$ onto $\kx^{D}_{Ds}$. We are particularly
interested in the \emph{plinth ideal} $\pl(D)$, that is, the ideal of $\kxd$
formed by the images $Ds$ of all local slices $s$ together with zero.

In Section \ref{section:what is separated}, we first observe that outside the zero set of any subset of
$\pl(D)$, the invariants separate the orbits (Proposition
\ref{prop:orbitSeparation}). This leads to a rather rough
description of the separating variety (Proposition
\ref{prop:SepVarDecomposition}): apart from the graph, the
separating variety is determined by the restrictions of the
invariants on the zero set of elements of $\pl(D)$. The last of our
results on arbitrary $\Ga$-actions is that if there is a polynomial
separating algebra, then the separating variety has no irreducible
component of dimension less than $\dim X+1=\dim \overline{\graph}$
(Proposition \ref{prop:PolySepAlg}).

In Section \ref{section:basic actions}, we focus further on the basic actions of the additive group, that
is, the finite dimensional indecomposable rational linear
representations of $\Ga$. We use the separating set constructed in
\cite{je-mk:sibga} to compute the separating variety and write it as
the union of irreducible components (Theorem \ref{thm:MainThm}). We
find that for $n$ odd, divisible by four, or equal to two, there is
exactly one irreducible component: the closure of the graph. On the
other hand, for $n>2$ even, but not divisible by four, we find a
second component. This component has smaller dimension than the
graph, which implies that there can not be
polynomial separating algebras (Corollary \ref{cor:noPolySepAlg}).

 Section \ref{section:ProofOfMainProp} contains the technical details for the
proof of the key result of Section \ref{section:basic actions}.

\begin{ack}

We thank Hanspeter Kraft and Gregor Kemper for their hospitality
during academic visits between the two authors.\end{ack}


\section{Separation properties of invariants}\label{section:what is separated}

Before we specialize to the basic actions of the additive group, we present
some general results on separating properties of invariants of additive group actions.
\begin{prop}\label{prop:orbitSeparation}
If $S\subseteq\sqrt{\pl(D)\kx}$,
then the invariants separate orbits outside $\V_X(S)$, that is,
\[
\SD\setminus (\V_{X}(S)\times\V_{X}(S))\subseteq \graphD.
\]

\begin{proof}
We may assume that $S\subseteq\pl(D)$. Suppose $x,y\in X\setminus\V_X(S)$ are not separated by any invariant, that
is, $f(x)=f(y)$ for all $f\in \kxd$. By our assumptions, there exist $f\in S$ and
$s\in \kx$ such that $f=D(s)$ and $f(x)=f(y)\ne 0$. Set $t_x=s(x)/f(x)\in \Ga$ and $t_y=s(y)/f(y)\in\Ga$. Suppose $\kx=\kk[a_1,\ldots,a_n]$. For each $a_i$, we have
\[
\begin{array}{l}
a_i((-t_x)* x)=(t_x* a_i)(x)=\left(\sum_{k=0}^\infty\frac{(-t_{x})^{k}}{k!}D^{k}(a_{i})\right)(x)\\
~~~~~~~~~~~~=\left(\sum_{k=0}^\infty\frac{(-1)^{k}s(x)^{k}}{(f(x))^{k}k!}D^{k}(a_{i})\right)(x)=\left(\sum_{k=0}^\infty\frac{(-1)^{k}s^{k}}{(Ds)^{k}k!}D^{k}(a_{i})\right)(x).\end{array}
\]
By the Slice Theorem (see \cite[Proposition
2.1]{ae:aacigaav}), $\sum_{k=0}^\infty\frac{(-1)^{k}s^{k}}{(Ds)^{k}k!}D^{k}(a_{i})$ is in $\kx^{D}_f$, and as $x$ and $y$ are not separated by invariants (including $f$), it follows that
\[
a_i((-t_x)* x)=a_i((-t_y)* y) \quad \text{ for }i=1,\ldots,n,
\]
that is, $(-t_x)* x=(-t_y)* y$, and so $x$ and $y$ are in the same orbit.
\end{proof}
\end{prop}

Note that if $S\subseteq \kx$ consists of $\Ga$-invariants, then 
$\V_{X}(S)$  is $\Ga$-stable. In particular, the following Proposition
is a first step of the decomposition of the separating variety in
$\Ga\times \Ga$-stable subsets:

\begin{prop}\label{prop:SepVarDecomposition}
Let $I\subseteq \sqrt{\pl(D)\kx}$ be an ideal of $\kx$, and consider
the canonical projection $\tau: \kx\rightarrow\kx/I$, given by $f\mapsto
 f+I$. Let
$A\subseteq \kxd$ be a separating algebra. If $h_{1},\ldots,h_{r}$ are elements
of $\kx$ such that $\kk[\tau(h_{1}),\ldots,\tau(h_{r})]=\tau(A)$, then the separating variety decomposes as
\[
\SD=\Big( \V_{X\times X}(\delta(h_{1}),\ldots,\delta(h_{r}))\cap
  \left(\V_{X}(I)\times \V_{X}(I)\right)\Big) \cup \graphD .
\]

\begin{proof}~\\
\noindent
``$\supseteq$'': We have already seen that $\SD\supseteq
\graphD$. Take a point $(x,y)\in \V_{X\times X}(\delta(h_{1}),\ldots,\delta(h_{r}))\cap
  \left(\V_{X}(I)\times \V_{X}(I)\right)$. We have to show $(x,y)\in\SD$,
  that is, $f(x)=f(y)$ for all $f\in\kxd$. As $A$ is a separating algebra, it
  suffices to show that $f(x)=f(y)$ for all $f\in A$. Let $f$ be an element of
  $A$. As $\tau(f)$ is in $\kk[\tau(h_{1}),\ldots,\tau(h_{r})]$, there exists
  a polynomial $p$ in $r$ variables such that
  $\tau(f)=p(\tau(h_{1}),\ldots,\tau(h_{r}))$. Therefore,
  $f-p(h_{1},\ldots,h_{r})\in I$. As $(x,y)$ is an
  element of $\V_{X}(I)\times \V_{X}(I)$, we have $g(x)=g(y)=0$ for all $g\in
  I$.  In particular, we have
\[
f(x)-p(h_{1}(x),\ldots,h_{r}(x))=0=f(y)-p(h_{1}(y),\ldots,h_{r}(y)).
\]
As also $(x,y)\in \V_{X\times
  X}(\delta(h_{1}),\ldots,\delta(h_{r}))$, we have $h_{i}(x)=h_{i}(y)$ for
  $i=1,\ldots,r$. It follows that $f(x)=f(y)$.

\noindent
``$\subseteq$'': It suffices to show that $\SD\setminus\graphD \subseteq  \V_{X\times X}(\delta(h_{1}),\ldots,\delta(h_{r}))\cap
  \left(\V_{X}(I)\times \V_{X}(I)\right)$. Take $(x,y)\in
\SD\setminus\graphD$. By Proposition \ref{prop:orbitSeparation}, we have that $(x,y)\in\V_{X}(I)\times \V_{X}(I)$. It remains to show that $(x,y)\in
\V_{X\times X}(\delta(h_{1}),\ldots,\delta(h_{r}))$. Take elements
$g_{i}\in A$ such that $\tau(h_{i})=\tau(g_{i})$ for $i=1,\ldots,r$. There
then exist elements $q_{i}\in I$ such that $h_{i}=g_{i}+q_{i}$ for
$i=1,\ldots,r$. As $x,y\in \V_{X}(I)$, we have $q_{i}(x)=0=q_{i}(y)$ for all
$i$. As $(x,y)\in \SD$, we have $g_{i}(x)=g_{i}(y)$ for all $i$. Hence,
\[
h_{i}(x)=g_{i}(x)+q_{i}(x)=g_{i}(y)+q_{i}(y)=h_{i}(y)\quad\text{ for all } i=1,\ldots,r.
\]
This shows that $\delta(h_{i})(x,y)=0$ for $i=1,\ldots,r$, and so we are done.
\end{proof}
\end{prop}

\begin{eg}
 We consider Daigle and Freudenburg's 5-dimensional
 counterexample to Hilbert's fourteenth problem (see \cite{dd-gf:achfpd5}).  Let $X:=\kk^5$ and let  $R:=\kk[x,s,t,u,v]$ be the ring of regular functions on $X$. Define a LND on $R$ via
 \[\Delta:=x^3\frac{\partial}{\partial s}+s\frac{\partial}{\partial t}+t\frac{\partial}{\partial u}+x^2\frac{\partial}{\partial v}.\]
In  \cite{ed-mk:afssdfchfp}, we constructed the following separating algebra
for $R^{\Delta}$:
 \[\begin{array}{rl}
 A:=&\kk[f_1,f_2,f_3,f_4,f_5,f_6]\\
    =&\kk[x, \,2x^3t-s^2, \,\,3x^6u-3x^3ts+s^3, \,\,xv-s,\,\,\, x^2ts-s^2v+\\
 &\hspace{.7cm} 2x^3tv-3x^5u, -18x^3tsu+9x^6u^2+8x^3t^3+6s^3u-3t^2s^2].
 \end{array}\]
 We have $x^3=\Delta(s)\in R^\Delta$ and $f_2=2x^3t-s^2=\Delta(3x^3u-st)\in
 R^\Delta$, thus, $x$ and $s$ are in $\sqrt{\pl(\Delta)R}$.  By
 \cite[Proposition 3.2]{ed-mk:afssdfchfp}, we have ${R}^{\Delta}\subseteq \kk\oplus(x,s){R}$. Thus, if $\tau:R\rightarrow R/(x,s)R$ is the canonical projection, for any separating algebra $A\subseteq R^\Delta$, we  have $\kk[\tau(1)]=\tau(A)$. By Proposition~\ref{prop:SepVarDecomposition}, it follows that
  \begin{equation}\label{eqn:sdDF5}
  \s_\Delta=\big(\V_X(x,s)\times\V_X(x,s)\big) \cup \overline{\Gamma_\Delta}.\end{equation}
  Both sets on the right hand side are irreducible and of dimension~$6$, and
  one can check  (using Magma \cite{Magma}, for example) that neither contains the
  other. Therefore, Equation \eqref{eqn:sdDF5} gives $\s_\Delta$ as the union of irreducible components.\done
\end{eg}

\begin{rmk}
One can compute the separating variety in a similar manner for Roberts'
counterexample \cite{pr:aigsbpsrnchfp} and the derivation investigated in
\cite[section 5]{ed:siqq}. Indeed, in both cases there is $S\subseteq\sqrt{\pl(D)}$ such that $\kxd\subseteq \kk+S\kx$.
\end{rmk}

As the separating variety contains the graph, its dimension is at least that
of the graph. It can be bigger, as in \cite[Example 2.1]{gk:cirgpc} and
Example~\ref{eg:F6} below, and as we can see from Theorem \ref{thm:MainThm},
it can have components of smaller dimension. In characteristic zero, the
additive group has no non-trivial closed subgroups. Points are thus either
fixed or their stabilizer is trivial. When the $\Ga$-action is
non-trivial, the Zariski-closure of the graph therefore has dimension $\dim (X)+\dim(\Ga)=\dim (X)+1$ (see for example \cite[Section 10.3]{gk:acca}).

\begin{eg}\label{eg:F6}
We now consider Freudenburg's 6-dimensional counterexample to Hilbert's fourteenth problem (see \cite{gf:achfpd6}). Let $X:=\kk^6$ and let $B:=\kk[x,y,s,t,u,v]$ be the ring of regular functions on $X$. Define a LND on $B$ via:
\[D:=x^3\frac{\partial}{\partial s}+y^3s\frac{\partial}{\partial t}+y^3t\frac{\partial}{\partial u}+x^2y^2\frac{\partial}{\partial v}.\]
We have $D(s)=x^3\in B^D$ and $D(3x^3u-y^3st)=2x^3y^3t-y^6s^2\in B^D$, that is, $(x,ys)\subseteq\sqrt{\pl(D)B}$.  As $B^D\subseteq\kk\oplus(x,y)B$ (see \cite[Lemma 1]{gf:achfpd6}) and $B^D\subseteq \kk[y]\oplus(x,s)B$ (see \cite[Example 4.4]{ed:siqq}), if $\tau:B\rightarrow B/(x,ys)$ is the canonical projection, then for any separating algebra $A$, $\tau(A)\subseteq\kk[\tau(y)]$. By Proposition \ref{prop:SepVarDecomposition}, we have
\[\begin{array}{rl}
\SD  = &\overline{\graph} \cup \big( (\V_X(x,ys)\times\V_X(x,ys)) \cap \V_{X\times X}(y\otimes 1-1\otimes y) \big)\\
         = & \overline{\graph} \cup  \V_{X\times X}(x\otimes1,1\otimes x,y\otimes1, 1\otimes y) \cup \\
          & \hspace{1.8cm}\V_{X\times X}(x\otimes 1, 1\otimes x,s\otimes 1, 1\otimes s, y\otimes 1-1\otimes y).\end{array}\]
One can verify (again with Magma \cite{Magma}) that this gives us the separating variety as the union of three irreducible components of dimension $7$, $8$, and $7$, respectively.
This example also shows that in general, the dimension of the separating
         variety is not  $2\dim X-\dim(\kxd))$.\done
\end{eg}

\begin{prop}\label{prop:PolySepAlg}
If $D$ is nonzero and $\kxd$ admits a polynomial separating algebra, then
every irreducible component of $\SD$ has dimension at least $\dim
X+1$.

\begin{proof}
As $\kk$ has characteristic zero, any separating algebra $A$ has
field of fractions $Q(A)=Q(\kxd)$ (see \cite[Theorem 3.2.3]{ed:si},
or \cite[Proposition 2.3.10]{hd-gk:cit} when $\kxd$ is  finitely
generated). Thus a finitely generated separating algebra $A$ has
dimension $n:=\operatorname{trdeg}_\kk(Q(\kxd))=\dim X -1$ (see
\cite[Principle 11(e)]{gf:atlnd}). If $A$ is a polynomial ring, then
$A$ is generated by $n$ elements, say  $f_1, \ldots, f_{n}$.  It
follows that $\SD=\V_{X\times X}(\delta(A))=\V_{X\times
X}(\delta(f_1),\ldots,\delta(f_{n}))$ is cut out by $n$ elements. By
Krull's Principal Ideal Theorem (see for example, \cite[Theorem
10.2]{de:cavtag}), every irreducible component of $\SD$ has
codimension at most $n$, that is, dimension at least $\dim X+1$.
\end{proof}\end{prop}


\section{The basic actions}\label{section:basic actions}

We now concentrate on the basic actions of the additive group. They
are induced by the Weitzenb\"ock derivations $D_n=x_{0}\frac{\partial}{\partial
  x_{1}}+\ldots+x_{n-1}\frac{\partial}{\partial
  x_{n}}$ on the polynomial rings $\kk[x_{0},\ldots,x_{n}]=\kvn$. We recall some
results and notation from \cite{je-mk:sibga}, where separating sets for the
basic actions were first constructed. Define the invariants
\begin{equation*}
f_{m}:=\sum_{k=0}^{m-1}(-1)^{k}x_{k}x_{2m-k}+\frac{1}{2}(-1)^{m}x_{m}^{2}\in
\ker D_n \quad \text{ for } m=1,\ldots,\left\lfloor\frac{n}{2}\right\rfloor,
\end{equation*}
and $f_{0}:=x_{0}$.  For $m=0,\ldots,\lfloor\frac{n-1}{2}\rfloor$,
\cite[Equation (3)]{je-mk:sibga} also gives polynomials $s_{m}$ such
that $D_{n}s_{m}=f_{m}$. It follows that
\begin{equation*}
I_n:=(x_{0},\ldots,x_{\lfloor\frac{n-1}{2}\rfloor})=\sqrt{(f_{0},\ldots,f_{\lfloor\frac{n-1}{2}\rfloor})}
\subseteq \sqrt{\pl(D_{n})\kvn}.
\end{equation*}

Consider the projection $\tau: \kvn\rightarrow \kvn/I_{n}$. We can reformulate \cite[Proposition
3.1]{je-mk:sibga} as follows:
\begin{equation*}
\tau(\kvndn)=\left\{\begin{array}{cl}
\kk & \text{ for } 2\nmid n,\\
\kk[\tau(x_{m}^{2})]&\text { for } n=2m,~ 2\nmid m,\\
\kk[\tau(x_{m}^{2}),\tau(x_{m}^{3})]&\text { for } n=2m,~ 2\mid m.
\end{array}\right.
\end{equation*}
Proposition \ref{prop:SepVarDecomposition} then implies that the
separating variety $\SDn$ is
\begin{equation}\label{eqn:BasicDecomp}
\begin{array}{cl}
\big(\V_{V_n}(I_{n})\times \V_{V_n}(I_n)\big) \cup\overline{\graphDn}, & \text{ if }2\nmid n,\\
\big(\V_{V_n\times V_n}(\delta(x_{m}^{2}))\cap (\V_{V_n}(I_{n})\times \V_{V_n}(I_n))\big)\cup \overline{\graphDn}, & \text { if } n=2m,~ 2\nmid m,\\
\big(\V_{V_n\times V_n}(\delta(x_{m}))\cap (\V_{V_n}(I_n)\times \V_{V_n}(I_n))\big)\cup \overline{\graphDn}, & \text { if } n=2m,~ 2\mid m.
\end{array}
\end{equation}

We formulate the technical part of our main result as a proposition whose proof is postponed to Section \ref{section:ProofOfMainProp}:

\begin{prop}\label{prop:TheMainProp}~\\ \vspace{-0.3cm}
\begin{itemize}
\item[(a)] If $n=2m+1$ is odd, then \[\V_{V_n}(I_n)\times \V_{V_n}(I_n)\subseteq \overline{\graphDn}.\]
\item[(b)] If $n=2m$ is even, then

\begin{itemize}
\item[(i)] $(\V_{V_n}(I_{n})\times \V_{V_n}(I_n)) \cap \V_{V_n\times V_n}\left((x_{m}\otimes
1)-(-1)^{m}(1\otimes x_{m})\right)\subseteq \overline{\graphDn},$

\item[(ii)] and if furthermore $\kk=\CC$, then \[\overline{\graphDn}\setminus\graphDn\subseteq (\V_{V_n}(I_n)\times \V_{V_n}(I_n)) \cap \V_{V_n\times V_n}\left((x_{m}\otimes
1)-(-1)^{m}(1\otimes x_{m})\right).\]
\end{itemize}\end{itemize}
\end{prop}

\begin{thm} \label{thm:MainThm} ~\\ \vspace{-0.3cm}
\begin{itemize}
\item[(a)] If $n$ is odd, divisible by four, or equal to $2$, then the separating variety is equal to the Zariski closure of the graph of the $\Ga$-action, that is, $\SDn= \overline{\graphDn}$.
\item[(b)] If $n=2m$ and $m\geq 3$ is odd, then the separating variety has two
  irreducible components:
  \begin{itemize}
  \item $ \overline{\graphDn}$, which has dimension $n+2$,
  \item and a second of dimension $n+1$:
\[
\V_{V_{n}\times V_n}\big(x_{m}\otimes 1-1\otimes x_{m}\big)\cap \big(\V_{V_n}(I_n)\times \V_{V_n}(I_n)\big).
\]
\end{itemize}\end{itemize}

\begin{proof}
(a) If $n$ is odd or  divisible by 4, the claim follows
immediately from Equation \eqref{eqn:BasicDecomp} and Proposition \ref{prop:TheMainProp} (a) and (b)(i), respectively.

When $n=2$, Equation \eqref{eqn:BasicDecomp} gives
${\mathcal S}_{D_{2}}=M \cup \overline{\Gamma_{D_{2}}}$, where
\[
M=\{((0,a_{1},a_{2}),(0,b_{1},b_{2}))\in V_{2}\times V_{2} \mid
a_{i},b_{i}\in\kk,\,\, a_{1}^{2}=b_{1}^{2}\}.
\]
Take $(a,b)\in M$. If $a_{1}=b_{1}\ne 0$, then setting $t=\frac{b_{2}-a_{2}}{a_{1}}$, we obtain
\[
t*a=(0,a_{1},a_{2}+ta_{1})=(0,b_{1},b_{2})=b,
\]
and so $(a,b)\in\Gamma_{D_{2}}$. On the other hand, if $a_{1}=-b_{1}$, then Proposition \ref{prop:TheMainProp}(b)(i)
implies $(a,b)\in\overline{\Gamma_{D_{2}}}$.

\noindent
(b) Assume $n=2m$ with $m\ge 3$ odd. Equation \eqref{eqn:BasicDecomp} yields
$\SDn=\overline{\graphDn}\cup M_{n,1}\cup M_{n,2}$, where $M_{n,i}$ is the set of points
of $V_n\times V_n$ of the form
\[
((0,\ldots,0,a_{m},a_{m+1},\ldots,a_{2m}),(0,\ldots,0,(-1)^{i}a_{m},b_{m+1},\ldots,b_{2m}))\]
for $i=1,2$, and where $a_{k},b_{k}\in\kk$. By Proposition
\ref{prop:TheMainProp} (b)(i), we have
$M_{n,1}\subseteq\overline{\graphDn}$, and so
$\SDn=\overline{\graphDn}\cup M_{n,2}$. We clearly have
$\graphDn\not\subseteq M_{n,2}$. It remains to show that
$M_{n,2}\not\subseteq \overline{\graphDn}$.

The $\Ga$-actions we consider are in fact defined over $\QQ$. Thus, $\overline{\graphDn}$ is the zero set of an ideal generated by polynomials with coefficients in $\QQ$ (often called the Derksen-ideal, see \cite{hd:coi,  gk:cirgpc}). Clearly, this also holds for $M_{n,2}$.  Note that ideal inclusion can be decided using Gr\"obner Basis methods. Hence, the question of the inclusion of $M_{n,2}$ in  $\overline{\graphDn}$ will have the same answer over any field of characteristic zero, and we may assume $\kk=\CC$. Suppose, for a contradiction, that $M_{n,2}\subseteq \overline{\graphDn}$. Proposition
\ref{prop:TheMainProp}(b)(ii) then implies that $M_{n,2}\setminus
{\graphDn}\subseteq{\overline{\graphDn}}\setminus{\graphDn}\subseteq
{M_{n,1}}$. As $m\ge 3$, this is a contradiction. Indeed, if
$a=e_{m}$ and $b=e_{m}+e_{m+1}$ (where $e_{0},\ldots,e_{n}$ are the
standard basis vectors), then $(a,b)\in M_{n,2}\setminus\graphDn$,
but $(a,b)\not\in M_{n,1}$. 
\end{proof}
\end{thm}

\begin{cor}\label{cor:noPolySepAlg}
If $n=2m$ and $m\geq 3$ is odd, then $\kvndn$ does not admit a polynomial
separating algebra.
\begin{proof}
Follows directly from Theorem \ref{thm:MainThm}(b) and Proposition \ref{prop:PolySepAlg}.
\end{proof}
\end{cor}


\section{Proof of Propostion \ref{prop:TheMainProp}}\label{section:ProofOfMainProp}

We first prove a technical lemma, which in turn uses the following
well-known formula (see, for example \cite[Satz 1.25]{hrh-wh:ek}):
\begin{equation}\label{eqn:sumFormula}
\sum_{j=0}^{r}(-1)^{j}{r\choose j}{p-j \choose q}={p-r \choose p-q}\quad \text
{ for all }p\ge q\ge 0,~ r\ge 0.
\end{equation}

\begin{lem}\label{lem:TheMatrix}~\\ \vspace{-0.3cm}

\begin{enumerate}
\item[(a)] If $2m\le n$ are natural numbers, then $M_{m,n}=\left(\frac{1}{(n-i-j)!}\right)_{i,j=0,\ldots,m}\in\QQ^{(m+1)\times (m+1)}$ is an invertible matrix.

\item[(b)] For any  $m$, if
$A:=M_{m-1,2m}=\left(\frac{1}{(2m-i-j)!}\right)_{i,j=0,\ldots,m-1}\in\QQ^{m\times m}$ and
$v:=(\frac{1}{m!},\frac{1}{(m-1)!}\ldots,\frac{1}{1!})^{T}\in\QQ^{m}$, then $v^{T}A^{-1}v=1-(-1)^{m}$.
\end{enumerate}
\begin{proof}~\\
\noindent
(a) For each $i=0,\ldots,m$, multiply the $i$th line of $M_{m,n}$ by
$(n-i)!$ to obtain the matrix $\left(\frac{(n-i)!}{(n-i-j)!}\right)_{i,j=0,\ldots,m}$
 having the same rank as $M_{m,n}$. This is the evaluation matrix
 $(f_{j}(a_{i}))_{i,j=0,\ldots,m}$ of the polynomials
 $f_{j}=X(X-1)\cdots(X-j+1)$ of degree $j$ at the points $a_{i}=n-i$,  and thus, it is invertible.

\noindent
(b) Set
\[ x=\left(\frac{(-1)^{m+j}(2m-j+1)!}{(j-1)!(m-j+1)!}\right)_{j=1,\ldots,m}\in\QQ^{m}.\]
We first  show that $Ax=v$. For $i=1,\ldots,m$, we have to show that
\[
\sum_{j=1}^{m}\frac{(-1)^{m+j}(2m-j+1)!}{(2m-i-j+2)!(j-1)!(m-j+1)!}=\frac{1}{(m-i+1)!},
\]
which is equivalent to
\[
\sum_{j=1}^{m+1}\frac{(-1)^{m+j}(2m-j+1)!}{(2m-i-j+2)!(j-1)!(m-j+1)!}=0.
\]
The left-hand side is equal to
\begin{eqnarray*}
&&(-1)^{m+1}\sum_{j=0}^{m}\frac{(-1)^{j}(2m-j)!}{(2m-i-j+1)!j!(m-j)!}\\
 &=&(-1)^{m+1}\frac{(i-1)!}{m!}\underbrace{\sum_{j=0}^{m}(-1)^{j}{2m-j\choose
i-1}{m\choose j}}_{(*)}.
\end{eqnarray*}
Formula \eqref{eqn:sumFormula} with $r:=m$, $p:=2m$, $q:=i-1$ implies that the
sum $(*)$ is equal to ${m\choose 2m-i+1}$, which is zero for $i=1,\ldots,m$, and so $Ax=v$.

Next, we show that $v^{T}x=1-(-1)^{m}$, that is,
\[
\sum_{j=1}^{m}\frac{(-1)^{m+j}(2m-j+1)!}{(j-1)!((m-j+1)!)^{2}}=1-(-1)^{m},
\]
or again
\[
\sum_{j=0}^{m}\frac{(-1)^{j}(2m-j)!}{j!((m-j)!)^{2}} =
\sum_{j=0}^{m}(-1)^{j}{m\choose j} {2m-j\choose m}=1.\] Since
Formula \eqref{eqn:sumFormula} with $r=m$, $p=2m$, and $q=m$ yields
the last equality, we have shown that $v^TA^{-1}v=v^Tx=1-(-1)^m$.
\end{proof}
\end{lem}

\begin{proof}[Proof of Proposition \ref{prop:TheMainProp}]
Set $m:=\lfloor\frac{n}{2}\rfloor$. We start by reformulating the
three statements:
\begin{itemize}
\item[(a)] Suppose $n=2m+1$ is odd. If
 \begin{equation}\label{eqn:formABodd}
\begin{array}{rl}
&a=(0,\ldots,0,a_{m+1},\ldots,a_{n})\\
\text{ and } &b=(0,\ldots,0,b_{m+1},\ldots,b_{n}),\end{array}\end{equation}
then  $(a,b)\in\overline{\graphDn}$.

\item[(b)(i)] Suppose $n=2m$ is even. If
\begin{equation}\label{eqn:formAbeven}
\begin{array}{rl}
&a=(0,\ldots,0,a_{m},a_{m+1},\ldots,a_{n})\\\text{  and }&
b=(0,\ldots,0,(-1)^{m}a
_{m},b_{m+1},\ldots,b_{n}),\end{array}\end{equation}
then $(a,b)\in\overline{\graphDn}$.

\item[(b)(ii)] If $\kk=\CC$, then every point $(a,b)\in\overline{\graphDn}\setminus\graphDn$ is of the form given in \eqref{eqn:formAbeven}.
\end{itemize}

We prove (a) and (b)(i) simultaneously by constructing a morphism
\[\begin{array}{lrcl}
f: & \kk  & \longrightarrow &V\times V\\
   & u     & \longmapsto     & (x(u),y(u)),\end{array}\]
   such that
   \begin{itemize}
   \item $f(0)=(a,b)$, as given in Equation \eqref{eqn:formABodd} or
\eqref{eqn:formAbeven}, if $n$ is odd or even, respectively,
   \item and  for each $u\neq 0$, we have $y(u)=\nicefrac{1}{u}*x(u)$.
   \end{itemize}

As $\overline{\graphDn}$ is Zariski-closed, $f^{-1}(\overline{\graphDn})$ is also Zariski-closed and will then contain $\kk\setminus \{0\}$. Thus, $ f^{-1}(\overline{\graphDn})$ must contain $\kk$, and in particular, we will have $(a,b)=f(0)\in \overline{\graphDn}$.

Set $m':=\lfloor\frac{n-1}{2}\rfloor$, so that for $n$ odd,  we have $m'=m$, and for $n$
 even, $m'=m-1$. Note that $n=(m+1)+m'$ in both cases. We impose the following restrictions:
 \[\begin{array}{rl}
 x(u)  := & \underbrace{(x_0(u),\ldots,x_{m'}(u)}_{:=\tilde{x}(u)},a_{m'+1},\ldots,a_n),\\
 y(u)  := & \underbrace{(y_0(u),\ldots,y_{m}(u)}_{:=\tilde{y}(u)},b_{m+1},\ldots,b_n).
 \end{array}\]
 For $z=(z_{i})_{i=0,\ldots,n}\in V_{n}$ and $t\in\Ga$, the group action
is given by
\[
t*z=\left(\sum_{i=0}^{k}\frac{t^{k-i}}{(k-i)!}z_{i}\right)_{k=0,\ldots,n}.
\]
If $x(u)$ and $y(u)$ define a morphism as desired, for $u\neq 0$, we must have $(\nicefrac{1}{u}*x(u))_{k=m+1,\ldots,n}=(b_{k})_{k=m+1,\ldots,n}$, or equivalently
\[
\sum_{i=0}^{m'}\frac{1}{u^{k-i}(k-i)!}x_{i}+\sum_{i=m'+1}^{k}\frac{1}{u^{k-i}(k-i)!}a_{i}=b_{k}
\quad \text{ for }k=n,\ldots,m+1.
\]
Set $\delta:=m-m'$, so that $\delta=0$ for $n$ odd, and $\delta=1$
for $n$ even, then $\tilde{x}(u)$ must be a solution of the
following system of linear equations:
\begin{equation}\label{eqn:MatrixFormeq}
\underbrace{\left(\begin{array}{cccc}
\frac{1}{u^n n!}&\frac{1}{u^{n-1}(n-1)!}&\dots&\frac{1}{u^{m+1}(m+1)!}\\
\frac{1}{u^{n-1}(n-1)!}&\ddots&&\vdots\\
\vdots&&\ddots&\vdots\\
\frac{1}{u^{m+1}(m+1)!}&\frac{1}{u^m m!}&\dots&\frac{1}{u^{\delta+1} (\delta+1)!}
\end{array}\right)}_{=:C(u)}\left(\begin{array}{c}x_{0}\\x_{1}\\\vdots\\x_{m'}\end{array}
\right)=\left(\begin{array}{c}p_{m}\\p_{m-1}\\\vdots\\p_{\delta}\end{array}
\right),
\end{equation}
where
$p_{k-m'-1}(u):=b_{k}-\sum_{i=m'+1}^{k}\frac{(\nicefrac{1}{u})^{k-i}}{(k-i)!}a_{i}$
for $k=n,\ldots,m+1$.

Observe that
\[C(u)=u^{\delta-1}\cdot\diag(u^{-m},u^{1-m},\ldots,u^{-\delta})\cdot A\cdot
\diag(u^{-m},u^{1-m},\ldots,u^{-\delta}),\] where $A:=M_{m',n}$ is
the invertible matrix of Lemma \ref{lem:TheMatrix}(a). Thus, for
nonzero $u$, we must have
\[
\tilde{x}(u)=\left(\begin{array}{c}x_{0}\\x_{1}\\\vdots\\x_{m'}\end{array}
\right)=u^{1-\delta}\diag(u^{m},u^{m-1},\ldots,u^{\delta})A^{-1}
\underbrace{\left(\begin{array}{c}u^{m}p_{m}\\u^{m-1}p_{m-1}\\\vdots\\u^{\delta}p_{\delta}\end{array}
\right)}_{=:q(u)}.
\]
Note that above, $q_k(u)=u^{k}p_{k}(u)$ is a polynomial in $u$, thus
for this choice of $\tilde{x}(u)$, we obtain a morphism satisfying
$x(0)=a$ as desired. For $u\neq 0$ and $k=0,\ldots,m'$, we must then
have
\begin{eqnarray*}y_{k}(u)&:=&\sum_{i=0}^{k}\frac{(1/u)^{k-i}}{(k-i)!}x_{i}(u)=\left(\frac{u^{-k}}{k!},\frac{u^{-k+1}}{(k-1)!},\ldots,\frac{1}{0!},0,\ldots,0\right)
\tilde{x}(u)\\
&=&u^{1-\delta}\left(\frac{u^{m-k}}{k!},\frac{u^{m-k}}{(k-1)!},\ldots,\frac{u^{m-k}}{0!},0,\ldots,0\right)A^{-1}q(u).
\end{eqnarray*}
This gives an expression of $y_k(u)$ as a polynomial in $u$. For $k=0,\ldots,m'$, we have $y_{k}(0)=0$. For $n$ odd, it already follows that $y(0)=b$, and we are done. If $n$ is even, then
\begin{equation}\label{eqn:yu} y_{m}(u)=a_{m}+\left(\frac{1}{m!},\frac{1}{(m-1)!},\ldots,\frac{1}{2!},\frac{1}{1!}\right)A^{-1}q(u),
\end{equation}
which is again polynomial in $u$.  It remains to
show that $y_{m}(0)=(-1)^{m}a_{m}$ for $n=2m$. As $q_{k}(0)=-\frac{a_{m}}{k!}$ for $k=\delta,\delta+1,\ldots,m$,
formula~\eqref{eqn:yu} yields
\begin{eqnarray*}
y_{m}(0)&=&a_{m}+\left(\frac{1}{m!},\frac{1}{(m-1)!},\ldots,\frac{1}{2!},\frac{1}{1!}\right)A^{-1}\left(\begin{array}{c}-\frac{a_{m}}{m!}\\-\frac{a_{m}}{(m-1)!}\\\vdots\\ -\frac{a_{m}}{1!}
    \end{array}\right)\\
&=&a_{m}(1-v^{T}A^{-1}v)=(-1)^ma_m,
\end{eqnarray*}
where $v=(\frac{1}{m!},\frac{1}{(m-1)!}\ldots,\frac{1}{1!})^{T}\in\kk^{m}$, and the last equality follows from Lemma \ref{lem:TheMatrix}(b).

We now prove (b)(ii). Recall that for a constructible subset $U$ in
an affine complex variety, the Zariski-closure coincides with the
closure taken in the Euclidean topology (see \cite[Satz
11.23]{mb:agee}). In particular, as images of morphisms are
constructible, this result holds for the image $\graphDn$ of the
graph morphism $\phi:\Ga\times V_{n}\rightarrow V_{n}\times V_{n}$
defined by $\phi(t,x)=(x,t*x)$. Let
$(a,b)\in\overline{\graphDn}\setminus\graphDn$. We must show that
$(a,b)$ is of the form described  in \eqref{eqn:formAbeven}. By
Proposition \ref{prop:orbitSeparation}, $(a,b)\in{\mathcal
S}_{D_{n}}\setminus\graphDn$ implies that
$(a,b)\in\V_{V_{n}}(I_{n})\times\V_{V_{n}}(I_{n})$, that is,
$a=(0,\ldots,0,a_{m},a_{m+1},\ldots,a_{n})$ and
$b=(0,\ldots,0,b_{m},b_{m+1},\ldots,b_{n})$. Thus, it remains to
show that  $b_{m}=(-1)^{m}a_{m}$. As $(a,b)\in\overline{\graphDn}$,
there exists a sequence $(t_{l},x^{l})_{l\in\NN}\in (\Ga\times
V_{n})^{\NN}$ such that
\[\lim_{l\rightarrow\infty}(x^{l},t_{l}*x^{l})=(a,b).\]
If the sequence $(t_{l})_{l\in\NN}$ was bounded, there would be a
convergent subsequence with limit $t'$, and as the group action is a
continuous map, we would have $(a,b)=(a,t'*a)\in\graphDn$, a
contradiction. Thus, the sequence $(t_{l})_{l\in\NN}$ is unbounded,
and we can assume
\[
\lim_{l\rightarrow \infty}t_{l}=\infty \quad \text{ and }t_{l}\ne 0 \text{ for
all }l\in\NN.
\]
 Set $y^{l}:=t_{l}*x^{l}$, and write
$x^{l}=(x_{0}^{l},x_{1}^{l},\ldots,x_{n}^{l})$ and
$y^{l}=(y_{0}^{l},y_{1}^{l},\ldots,y_{n}^{l})$. We then have the following
equations:
\[y_{k}^{l}=\sum_{i=0}^{k}\frac{t_{l}^{k-i}}{(k-i)!}x_{i}^{l}\quad\quad\text{
  for }k=0,\ldots,n.
\]
For each $k=m+1,\ldots,2m$, these equations can be written as
\[
\sum_{i=0}^{m-1}\frac{t_{l}^{k-i}}{(k-i)!}x_{i}^{l}=\underbrace{y_{k}^{l}-\sum_{i=m}^{k}\frac{t_{l}^{k-i}}{(k-i)!}x_{i}^{l}}_{=:p^{l}_{k-m}}.
\]

As in Equation \eqref{eqn:MatrixFormeq}, we may write
these equations in matrix form:
\begin{equation*}
\underbrace{\left(\begin{array}{cccc}
\frac{t_{l}^{2m}}{2m!}&\frac{t_{l}^{2m-1}}{(2m-1)!}&\dots&\frac{t_{l}^{m+1}}{(m+1)!}\\
\frac{t_{l}^{2m-1}}{(2m-1)!}&\ddots&&\vdots\\
\vdots&&\ddots&\vdots\\
\frac{t_{l}^{m+1}}{(m+1)!}&\frac{t_{l}^{m}}{m!}&\dots&\frac{t_{l}^{2}}{2!}
\end{array}\right)}_{=:C_{l}}\left(\begin{array}{c}x_{0}^{l}\\x_{1}^{l}\\\vdots\\x_{m-1}^{l}\end{array}
\right)=\left(\begin{array}{c}p_{m}^{l}\\p_{m-1}^{l}\\\vdots\\p_{1}^{l}\end{array}
\right).
\end{equation*}
If
$A:=M_{m-1,2m}=\left(\frac{1}{(2m-i-j)!}\right)_{i,j=0,\ldots,m-1}\in\QQ^{m\times m}$
is the invertible matrix of Lemma \ref{lem:TheMatrix}(a), then
\[C_{l}=\diag(t_{l}^{m},t_{l}^{m-1},\ldots,t_{l})\cdot A\cdot
\diag(t_{l}^{m},t_{l}^{m-1},\ldots,t_{l}),\]
and thus,
\[
\left(\begin{array}{c}x^{l}_{0}\\x^{l}_{1}\\\vdots\\x^{l}_{m-1}\end{array}
\right)=\diag(t_{l}^{-m},t_{l}^{-m+1},\ldots,t_{l}^{-1})A^{-1}\left(\begin{array}{c}t_{l}^{-m}p_{m}\\t_{l}^{-m+1}p_{m-1}\\\vdots\\t_{l}^{-1}p_{1}\end{array}
\right).
\]
We define $q_{k}^{l}:=t_{l}^{-k}p_{k}$ for $k=1,\ldots,m$, and all $l$. We have
\begin{eqnarray*}
y_{m}^{l}&=&x_{m}^{l}+\sum_{i=0}^{m-1}\frac{t_{l}^{m-i}}{(m-i)!}x_{i}^{l} \nonumber \\
&=&x_{m}^{l}+\left(\frac{t_{l}^{m}}{m!},\frac{t_{l}^{m-1}}{(m-1)!},\ldots,\frac{t_{l}}{1!}\right)\left(\begin{array}{c}x^{l}_{0}\\x^{l}_{1}\\\vdots\\x^{l}_{m-1}\end{array}
\right)  \nonumber \\
&=&x_{m}^{l}+\left(\frac{1}{m!},\frac{1}{(m-1)!},\ldots,\frac{1}{1!}\right)A^{-1}\left(\begin{array}{c}q_{m}^{l}\\q^{l}_{m-1}\\\vdots\\q^{l}_{1}\end{array}
\right). \label{eqn:yml}
\end{eqnarray*}
For $k=1,\ldots,m$, we have
\begin{eqnarray*}
\lim_{l\to\infty}q_{k}^{l}&=&\lim_{l\to\infty}t_{l}^{-k}\left(y_{k+m}^{l}-\sum_{i=m}^{k+m}\frac{t_{l}^{k+m-i}}{(k+m-i)!}x_{i}^{l}\right)\\&=&\lim_{l\to\infty}\left(-\frac{1}{k!}x_{m}^{l}\right)=-\frac{a_{m}}{k!}.
\end{eqnarray*}
Therefore,  by Lemma \ref{lem:TheMatrix}(b),
\begin{eqnarray*}
b_{m}&=&\lim_{l\to \infty}y_{m}^{l}=a_m-a_m\cdot v^{T}A^{-1}v\\&=& a_m-a_m\cdot
(1-(-1)^{m})=(-1)^{m}a_{m},
\end{eqnarray*}
 as desired.
\end{proof}


\bibliographystyle{plain}
\bibliography{reference}
\end{document}